\documentclass{birkjour}
\usepackage[utf8]{inputenc}
\usepackage{graphicx}              % to include figures
\usepackage{amsmath}               % great math stuff
\usepackage{amsfonts}              % for blackboard bold, etc
\usepackage{amsthm}                % better theorem environments
\usepackage{amssymb}
\usepackage{shuffle}                 % for the symbol shuffle
\usepackage{tikz}  % to draw Young tableaux
\usepackage{ytableau}    %also to draw young tableaux
\usepackage{xcolor}
\usepackage{float}
\usepackage{mathrsfs}

\newtheorem{thm}{Theorem}[section]
\newtheorem{lem}[thm]{Lemma}
\newtheorem{prop}[thm]{Proposition}
\newtheorem{cor}[thm]{Corollary}

\newtheorem{defn}[thm]{Definition}
\newtheorem{rmk}[thm]{Remark}

\newcommand{\mpc}{\mathrm{MPCT}}
\newcommand{\wgt}{\mathrm{wt}}
\newcommand{\comp}{\mathrm{comp}}
\newcommand{\pct}{\mathrm{PCT}}
\newcommand{\spct}{\mathrm{SPCT}}
\newcommand{\smpc}{\mathrm{SMPCT}}
\newcommand{\qsym}{\mathrm{QSYM}}
\newcommand{\pc}{\mathcal{PC}}
\newcommand{\rw}{\mathrm{rw}}
\newcommand{\ts}{T_{S,\beta}}
\newcommand{\Des}{\mathrm{Des}}
\newcommand{\Peak}{\mathrm{Peak}}
\newcommand{\YQS}{\hat{\mathscr{S}}}
\newcommand{\RQS}{\check{\mathscr{S}}}

\title{ A Note on Jing and Li's type B quasisymmetric Schur Functions}
\author{Ezgi Kantarci O\u{g}uz}

\begin{document}
\ytableausetup{notabloids}

\maketitle

\begin{abstract} In 2015, Jing and Li defined type B quasisymmetric Schur functions and conjectured that these functions have a positive, integral and unitriangular expansion into peak functions. We prove this conjecture, and refine their combinatorial model to give explicit expansions in monomial, fundamental and peak bases. We also show that these functions are not quasisymmetric Schur, Young quasisymmetric Schur or dual immaculate positive, and do not have a positive multiplication rule.
    
\end{abstract}
\section{Introduction}

The algebra of quasisymmetric functions $\qsym$, introduced by Gessel in 1984 \cite{MR777705} is the terminal object in the category of Hopf Algebras \cite{MR2196760}, and has applications in representation theory \cite{MR1245159}, crystal graphs \cite{MR2527604} and matroids \cite{MR2552658}. In 2007, a new basis for $\qsym$ was defined by Haglund, Luoto, Mason and van Willigenburg \cite{MR2739497} called the quasisymmetric Schur functions. In addition to refining Schur functions in a natural way, these functions exhibit several nice properties of the Schur functions, including having Pieri and Littlewood-Richardson rules \cite{MR3097867}. 

A natural question to ask is if there is a type B analogue for the quasisymmetric Schur functions. In 2015, a candidate basis $\widehat{P}$ was proposed by Jing and Li \cite{MR3366479}, that refined Schur's $P$-functions as an alternating sum. They conjectured that these quasisymmetric Schur $P$-functions have an integral, unitriangular and positive expansion in terms of the peak functions. 

The main purpose of this work is to prove this conjecture, as well as give an explicit formula. We refine the peak composition tableaux defined in \cite{MR3366479} using the marked alphabet, giving precise expansions in terms of the monomial, fundamental and peak bases. We also provide counterexamples to show that those functions do not expand positively into Young's quasisymmetric Schur functions, and that their multiplication with Schur's $P$-functions is not $\widehat{P}$-positive.

\section{Preliminaries}
\subsection{Quasisymmetric Functions}

A \emph{weak composition} of $n$ is a list of non-negative numbers that add up to $n$, called its parts. A \emph{strong composition} is a weak composition whose parts are all non-zero. Throughout this work, we will be interested in strong compositions, which we will simply call the compositions of $n$, adding strong only when confusion is possible. For a composition $\beta$ of $n$, we will call $n$ the \emph{size} of $\beta$, denoted by $|\beta|$ and the number of its parts the \emph{length} of $\beta$, denoted by $\ell(\beta)$. For $n\leq m$, a composition $\gamma=(\gamma_1,\gamma_2,\ldots,\gamma_m)$ is said to be a \emph{refinement} of another composition $\beta=(\beta_1,\beta_2,\ldots,\beta_n)$ (denoted by $\gamma \preceq \beta$) if there exist numbers $i_1,i_2,\ldots,i_{n-1}$ such that we have $\beta_{j}=\sum_{k=i_{j-1}+1}^{i_j}\gamma_k$ for all $j\in[n]$, with the convention that $i_0=0$ and $i_n=m$.

Quasisymmetric functions are formal power series of bounded degree in variables $x_1,x_2,\ldots$ where for a given composition $\alpha=(\alpha_1,\alpha_2,\ldots,\alpha_n)$ the coefficient of $x_{i_1}^{a_1}x_{i_2}^{a_2}\cdots x_{i_t}^{a_t}$ is the same for any $i_1<i_2<\cdots<i_t$. They form a graded algebra $\qsym=\bigoplus_n \qsym^{(n)}$ where $\qsym^{(n)}$ stands for the homogeneous quasisymmetric functions of degree $n$. The subspace $\qsym^{(n)}$ is generated by the monomial quasisymmetric functions of degree $n$, defined for compositions $\beta=(\beta_1,\beta_2...\beta_m)$ of size $n$ as follows: 
\begin{eqnarray}
M_{\beta}(X)=\sum_{i_1< i_2<\cdots< i_m}x_{i_1}^{\beta_1}x_{i_2}^{\beta_2}\cdots x_{i_m}^{\beta_m}
\end{eqnarray}
where $X$ is a shorthand for the variables $(x_1,x_2,x_3,\ldots)$. 

Another basis for $\qsym$ that we will make use of is the fundamental basis, defined by Gessel \cite{MR777705}. Compositions of $n$ and subsets of $[n-1]$ are related by a bijection taking a composition $\beta=(\beta_1,\beta_2,\ldots,\beta_n)$ to the set $\{\beta_1, \beta_1+\beta_2,\ldots,\beta_1+\beta_2\cdots+\beta_{n-1}\}$. This bijection associates each $D \subseteq [n-1]$ to a unique composition, which we will denote by $\comp(D)$. Note that $\comp(D_1)$ is a refinement of $\comp(D_2)$  if and only if $D_1 \supseteq D_2$. The fundamental basis for quasisymmetric functions is given as follows:

\begin{eqnarray}\label{eq:F}
F_D(X)=\sum_{\beta \preceq \comp(D)} M_{\beta}(X).
\end{eqnarray}

\subsection{Peak Composition Tableaux and Quasisymmetric Schur $Q$-functions}

The compositions $\alpha=(\alpha_1,\alpha_2,\ldots,\alpha_m)$ of $n$ with no parts except possibly the last one equal to $1$ are called the \emph{peak compositions}.  We will use the notation $\pc(n)$ for the set of peak compositions of $n$ and set $\pc=\bigcup_n \pc(n)$. Note that for a subset $P$ of $[n-1]$, $\comp(P)$ is a peak composition if and only if $P$ does not contain any consecutive entries or $1$. 

For any $P \subseteq [n-1]$ with $\comp(P) \in \pc(n)$ we associate a quasisymmetric function called its \emph{peak function}, defined in \cite{MR1389788} by:
\begin{eqnarray}
G_P(X)=\sum_{\substack{D\in [n-1]\\ P \subseteq D \triangle D+1}}F_{D}(X),
\end{eqnarray}
where $D+1$ is the subset of $[n]\backslash\{1\}$ defined by adding $1$ to the elements of $D$, and $D \triangle D+1$ denotes the symmetric difference of the two sets.

The Ferrers diagram for a peak composition $\alpha=(\alpha_1,\alpha_2,\ldots,\alpha_k)$ is an array of boxes with $\alpha_i$ boxes in row $i$ (We use the English notation of drawing the first row highest in our diagrams).

For a given peak composition $\alpha$, a \emph{peak composition tableau} \cite{MR3366479} of shape $\alpha$ is a filling of the diagram of $\alpha$ with positive integers satisfying the following properties:

\begin{enumerate}
    \item Each row is weakly increasing from left to right.
    \item First column is strictly increasing from top to bottom.
    \item For any $k$, the subdiagram of $\alpha$ labeled with numbers $\{1, 2,\ldots,k\}$ is a peak composition.
\end{enumerate}

\begin{figure}[h]
      \begin{ytableau}
  1&2&6 \\
 3&4\\
 4&4&5\\
 5\\
  \end{ytableau}
  \caption{A peak composition tableau of shape (3,2,3,1) and weight (1,1,1,3,2,1)}
\end{figure}

 We denote by $\pct(\alpha)$ the set of peak composition tableaux of shape $\alpha$ whose weight is a strong composition (equivalently, no number is skipped in the labeling). To any $T \in \pct(\alpha)$ we can naturally associate a weak composition called the \emph{weight} of T, $\wgt(T)=(\wgt(T)_1,\wgt(T)_2,\ldots,\wgt(T)_n)$ where $\wgt(T)_i$ stands for the number of occurrences of $i$ in $T$. 
 
 \begin{figure}[h]
 $T_1$=     \begin{ytableau}
  1&1 \\
 2\\
  \end{ytableau}\qquad  \qquad    $T_2$=   \begin{ytableau}
  1&2 \\
 2\\
  \end{ytableau}\qquad    \qquad  $T_3$=   \begin{ytableau}
  1&2 \\
 3\\\end{ytableau}
     \caption{The elements of $\pct(2,1)$}
     \label{fig:pct21}
 \end{figure}
 
 We can extend the refinement ordering on compositions of a number $n$ to a total ordering called the \emph{lexicographic order} by setting $\alpha=(\alpha_1,\alpha_2,\ldots,\alpha_l) > \beta=(\beta_1,\beta_2,\ldots,\beta_m)$ if there exists some $k$ satisfying $\alpha_i=\beta_i$ for all $i <k$ and $\alpha_k > \beta_k$. 

\begin{rmk} \label{rmk:triangle}For a peak composition tableau $T$ of shape $\alpha$, we have $|\wgt(T)|=|\alpha|$. Furthermore, $\wgt(T)$ is less than or equal to $\alpha$ with the lexicographical order.
\end{rmk}

To facilitate our proofs, we will replace the original definition of quasisymmetric Schur $P$-functions with the defining Proposition $4.16$ from  \cite{MR3366479}.

\begin{defn}[\cite{MR3366479}] For a peak composition $\alpha$, the corresponding quasisymmetric Schur $P$- and $Q$-functions are defined as follows:
\label{def:jingli}
\begin{eqnarray}
\widehat{P}_{\alpha}(X)&=&\sum_{T\in\pct(\alpha)}2^{p(T)-m(T)}M_{\wgt(T)}(X),\\
\widehat{Q}_{\alpha}(X)&=&2^{\ell(\alpha)}\widehat{P}_{\alpha}(X).
\end{eqnarray}
Here $p(T)=\sum_{i\leq \ell(\alpha)}p_i(T)$ where $p_i(T)+1$ is the number of distinct integers in row $i$ and $m(T)$ denotes the number of boxes in the first column whose bottom and right adjacent boxes have the same number, as illustrated below:
\ytableausetup{boxsize=.42cm}
\begin{center}\begin{ytableau}
  x&y \\
 y\\
  \end{ytableau}\end{center}
\end{defn}

For the example  given in Figure \ref{fig:pct21}, we have:
\begin{eqnarray*}
\widehat{P}_{(2,1)}(X)= M_{(2,1)}(X)+M_{(1,2)}(X)+2M_{(1,1,1)}(X).
\end{eqnarray*}
Note that the multiplicity of the monomials corresponding to $T_1$ and $T_2$ are $1$, and the multiplicity of the monomial corresponding to $T_3$ is $2$. 

Schur's $Q$-functions are type $C$ analogues of the Schur functions, that are related to the projective representations of the symmetric group. The type $B$ case gives Schur's $P$-functions, which only differ from the corresponding $Q$-functions by a multiple of $2$. The interested reader can refer to Chapter $III.8$ of \cite{MR3443860} for more information on the subject.  It is shown in \cite{MR3366479} that Schur's $Q$-functions can be written as an alternating sum of quasisymmetric Schur $Q$-functions. 
\begin{thm}[\cite{MR3366479}] For any strict partition $\lambda$ we have:
\begin{eqnarray}\label{eq:alternating}
Q _{\lambda}(X)=\sum_{\substack{\alpha \in \pc\\ \mathrm{sort}(\alpha)=\lambda}}(-1)^{\mathrm{length}(\sigma_{\alpha})}\widehat{Q}_{\alpha}(X),
\end{eqnarray}
where $\mathrm{sort}(\alpha)$ is the unique rearrangement of $\alpha$ into a partition, and $\sigma_{\alpha}$ is a corresponding permutation of minimal length. 
\end{thm}

It was conjectured in \cite{MR3366479} that the quasisymmetric Schur $P$-functions have a unitary, unitriangular and positive expansion in terms of the peak functions. The main result of this paper is to provide a proof for this conjecture, and give this expansion explicitly (Theorem \ref{thm:main}).

\section{Marked Peak Composition Tableaux}

We introduce a variation on peak tableaux that will be a combinatorial model for the expansions in monomial, fundamental and peak bases.

\begin{defn}\label{def:mpc} For a peak composition $\alpha$, a \emph{marked peak composition tableau} of shape $\alpha$ is a filling of its diagram with numbers from the alphabet $1'<1<2'<2...$ satisying the following properties:
\begin{enumerate}
    \item Each row is weakly increasing from left to right, with no repeated marked numbers.
    \item First column is strictly increasing from top to bottom.
    \item For any positive number $k$, the subdiagram of $\alpha$ labeled with numbers $\{1', 1, 2', 2...k', k\}$ is a peak composition.
    \item If any box in the first column has an unprimed number $i$ in the box right adjacent to it, then it cannot have $i$ or $i'$ in the box adjacent below.
\end{enumerate}
\end{defn}

Note that erasing the marks from a marked peak composition tableau gives a peak composition tableau, so the definition of weight naturally extends to marked peak composition tableaux, with $\wgt(T)_i$ given by the total number of occurrences of $i$ and $i'$ in T.  

We will denote the set of the marked peak composition tableaux of shape $\alpha$ with weight given by a strong composition by $\mpc(\alpha)$. We will use the notation $\mpc^*(\alpha)$ to denote the subset of $\mpc(\alpha)$ where we have no marked numbers in the first column.

 \begin{figure}[h]
\begin{ytableau}
  1&1 \\
 2\\
  \end{ytableau}\qquad  \begin{ytableau}
  1&2' \\
 2\\
  \end{ytableau}\qquad \begin{ytableau}
  1&2' \\
 3\\\end{ytableau}\qquad \begin{ytableau}
  1&2 \\
 3\\\end{ytableau}
     \caption{The elements of $\mpc^*(2,1)$}
     \label{fig:mpc21}
 \end{figure}
 
 The elements of $\mpc^*(2,1)$ are listed in Figure \ref{fig:mpc21}. Note that by forgetting the markings, we have one element that maps to each of $T_1$ and $T_2$ each and two elements that map to $T_3$ from Figure \ref{fig:pct21}, so that we have:
 \begin{eqnarray*}
 \widehat{P}_{(2,1)}(X)&=&\sum_{T\in\mpc^*(2,1)}M_{\wgt(T)}(X).\\
 \end{eqnarray*}
 We will now see that this holds true in general.
 
\begin{prop} For a peak composition $\alpha$ we have:
\begin{eqnarray}
\widehat{P}_{\alpha}(X)&=&\sum_{T\in\mpc^*(\alpha)}M_{\wgt(T)}(X),\\
\widehat{Q}_{\alpha}(X)&=&\sum_{T\in\mpc(\alpha)}M_{\wgt(T)}(X).
\end{eqnarray}
\end{prop}
\begin{proof} To prove the statement for the P-functions, it suffices to show that for each peak composition tableau $T$, there are exactly $2^{p(T)-m(T)}$ ways of marking it to obtain a marked peak composition tableau with no marks in the first column, where $p(T)$ and $m(T)$ are defined as in Definition \ref{def:jingli}.

Note that if the $i$th row of $T$ contains more than one occurrence of a number $j$, only the leftmost can be marked. So for each distinct $j$ in row $i$, we get two choices, with two exceptions. First is, if the leftmost box labeled $j$ is in the first column, we are not allowed to mark it. So far we have $p_i(T)-1$ choices to make for each row. The second exception is, if the leftmost box in row $i$ has boxes labeled $j$ adjacent to the right and below, the one adjacent to the right cannot be marked- this situation occurs $m(T)$ times. Consequently we get exactly $2^{p(T)-m(T)}$ marked peak composition tableaux.

Case of $Q$-functions follows, as there are $2^{\ell(\alpha)}$ ways of marking the boxes in the first column, and the conditions of Definition \ref{def:mpc} is independent of the markings of the first column.
\end{proof}

\begin{defn}  A marked peak composition tableau of shape $\alpha \in \pc_n$  is called \emph{marked standard} if it contains each number from $1$ to $n$, possibly marked. A marked standard tableau with no marked entries is called \emph{standard}.\label{def:markedpct}
\end{defn}

\begin{figure}[h]
      \raisebox{-0.5cm}{$S_1=$} \begin{ytableau}
  1&2&8' \\
 3&4'\\
 5&6&7'\\
 9'\\
  \end{ytableau}
  \caption{A marked standard peak composition tableau of shape $(3,2,3,1)$}
  \label{fig:s0}
\end{figure}

We will denote the set of marked standard peak composition tableaux of shape $\alpha$ by $\smpc(\alpha)$, with $\smpc^*(\alpha)$ denoting the set of marked standard tableaux with no marked entries in the first column. We will denote the set of standard peak composition tableaux of shape $\alpha$ by $\spct(\alpha)$.

For a peak composition tableau $S$, we define its \emph{reading word} $rw(T)$ to be a reading of its numbers column by column, top to bottom, left to right.  For example, the tableau in Figure \ref{fig:s0} has reading word $\rw(S_1)=1359'24'68'7'$. If $S$ is a standard tableau, the reading word gives us a permutation of $n$ so we can talk about its descent set:
\begin{eqnarray*}
\Des(T)&=&\{i\mid i\text{ is to the right of }i+1\text{ in the reading word of }T\}.
\end{eqnarray*}
The definition of descent can be extended to the marked standard tableaux $S'$ reversing the marked coordinates and prepending them to the word, and then erasing marks. For example, if the tableau we are considering is $S_1$ from Figure \ref{fig:s0}, this operation gives us the word $784913526$. So $\Des(S_1)=\Des(784913526)=\{2,3,6\}$.

Note that for a marked standard tableau $T'$,  if $i$ comes after $i+1$ in the reading word of $T'$, then $i$ is a descent if and only if it is unmarked. If $i$ comes before $i+1$ in the reading word, then $i$ is a descent if and only if $i+1$ is marked. We can generalize this to the following lemma:

\begin{lem} Assume, $i$, $i+1$, \ldots, $i+k-1$ are not descents. If $i$ comes before $i+k$ in the reading word, then $i+k$ is not marked and if $i$ comes after $i+k$, then $i$ is marked. \label{thelemma}
\end{lem}
\begin{proof} If $k=1$, the result follows  directly from the definition of descent in marked words. Assume the statements hold for all $k'< k$, and let $k>1$. If $i+k$ is marked, by our induction hypothesis $i+1$ comes after $i+k$ in the reading word, and it is marked. Then as $i$ is not a descent, it cannot come before $i+1$ in the reading word, so it comes after $i+1$ and $i+k$. For the second part, assume $i$ is unmarked. Then by our induction hypothesis $i$ must come before $i+k-1$, and $i+k-1$ must be unmarked, which in turn implies $i+k$ must come after $i+k-1$ and $i$. 
\end{proof}

\begin{defn} Let $T$ be a marked peak composition tableaux of shape $\alpha$. Then $St'(T)$ is marked filling of the diagram of $\alpha$ given by assigning numbers $1, 2, 3... n$ to the boxes of $T$ in the following order,  and then marking $k$ if and only if the corresponding box is marked in $T$:
\begin{itemize}
    \item We assign numbers to the boxes in the order $1'<1<2'<2...$
   \item If there is more than one box of label $i$, the numbers increase following the order of the reading word.
   \item If there is more than one box of label $i'$, the numbers increase in the reverse order of the reading word.
\end{itemize}
\end{defn}

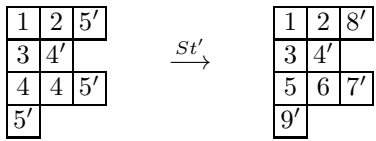
\begin{figure}[h]
      \begin{ytableau}
      1&2&5'\\
      3&4'\\
      4&4&5'\\
      5'\\
      \end{ytableau} \qquad \raisebox{-.4cm}{$\stackrel{St'}{\longrightarrow}$} \qquad
            \begin{ytableau}
      1&2&8'\\
      3&4'\\
      5&6&7'\\
      9'\\
      \end{ytableau} 
\end{figure}

\begin{prop} The filling $St'(T)$ defined above is a marked standard peak composition tableau.
\end{prop}
\begin{proof} By definition, $St'(T)$ contains each of the numbers $1'<1<2'<2...$ once, possibly marked. As each row in $T$ is weakly increasing, and any marked number $i'$ can only occur once in one row, the numbers in $St'(T)$ increase from left to right in each row. As the first column in $T$ is strictly increasing no $i$ or $i'$ is repeated, and the numbers in the first column of $St'(T)$ increase from bottom to top. Lastly, Condition (4) of Definition \ref{def:mpc} ensures that for any box in the first column, the number in its right adjacent box will be lower than the number in its bottom adjacent box, implying that the subdiagrams labeled with numbers less than some $k \leq n$ correspond to peak compositions.
\end{proof}

As $\Des(S') \subseteq [n-1]$, we can talk about the corresponding element $F_{\Des(S')}(X)$ of Gessel's fundamental basis, given in Equation $\ref{eq:F}$. In fact, these functions allow us to calculate the quasisymmetric Schur $Q$- and $P$-functions using only the marked standard peak composition tableaux.

\begin{thm} \label{thm:descent}The quasisymmetric Schur $Q$- and $P$-functions have a positive integral expansion in terms of $F_D$'s, given by:
\begin{eqnarray}
\label{eqn:defP} \widehat{P}_{\alpha}(X)&=&\sum_{S'\in\smpc^*(\alpha)}F_{\Des(S')}(X),\\
\label{eqn:defq} \widehat{Q}_{\alpha}(X)&=&\sum_{S'\in\smpc(\alpha)}F_{\Des(S')}(X).
\end{eqnarray}
\end{thm}
\begin{proof}We will prove the case for $Q$-functions only, the case for $P$-functions follows by duplication. By the definition of $F_D$, it suffices to show that for any $S\in \smpc(\alpha)$ and any $\beta \preceq \comp(\Des(S))$ (equivalently any $\beta$ satisfying $\Des(\beta) \subseteq \Des(S)$) , there is a unique tableau $T_{S,\beta}\in \mpc(\alpha,\beta)$ that has $St'(T_{S,\beta})=S$, and all peak composition tableaux of shape $\alpha$ are of this form for some $S$ and $\beta$.

First observe that as we standardize following the order of the numbers, for any marked standard peak composition tableau $S$ and composition $\beta=(\beta_1,\beta_2...\beta_n)$ if there exists a tableau $T$ that standardizes to $S$ with $\wgt(T)=\beta$, then $T$ must be equal to $T_{S,\beta}$, obtained by replacing $1,2..\beta_1$ in $S$ by $1$, $\beta_1+1,...\beta_1+\beta_2$ by $2$ and so on, and carrying the marks. We will denote  the box labeled $k/k'$ in $S$ by $C_k$ and its label in $T_{S,\beta}$ by $c_k$, and use the notation $|c_k|=i$ to mean $c_k\in\{i,i'\}$.

 The next step is to show that if $\Des(\beta) \supseteq \Des(S)$ then $T_{S,\beta} \in \mpc(\alpha,\beta)$ and $St'(\ts)=S$. We first verify that the conditions $(1)-(4)$ of the definition of marked peak composition tableaux (Definition \ref{def:mpc}) are satisfied. Note that condition $(3)$ is trivially satisfied as $\bigcup_{a\leq k}C_{a}$ is a peak composition for all $k$. Let $C_a$ and $C_{a+k}$ be in the same row. Then, as rows increase from left to right in $S$, $C_{a+k}$ is to the right of $C_a$, and by our construction, $|c_a| \leq |c_{a+k}|$. Furthermore if $|c_a| = |c_{a+k}|$, then $a, a+1...a+k-1 \notin \Des(\ts) \subseteq \Des(S)$. Also as $C_{a+k}$ is to the right of $C_a$, $a+k$ comes after $a$ in the reading word, so by Lemma \ref{thelemma} $a+k$ is not marked in $S$, implying $c_{a+k}$ is not marked in $\ts$. This takes care of $(1)$. The second condition is analogous to the first. For $(4)$, assume that we have a box in the first column with $C_a$ right adjacent and $C_b$ bottom adjacent to it with $|c_a|=|c_b|=i$. We want to show that in this case we mush have $c_a=i'$. First note that as $\bigcup_{k < r}C_{k}$ is a peak composition, it cannot include two parts of size $1$, so we must have $b>a$. That means, if $c_a=i$ then $c_b$ is also equal to $i$, implying $a, a+1,\ldots,b-1$ are not descents. As $C_b$ comes before $C_a$ in the reading word, by Lemma \ref{thelemma} $c_a$ must be marked, contradicting the assumption $c_a=i$. This shows $\ts$ is a valid marked peak composition tableau, its standardization is well defined. Let us show that its standardization is indeed $\ts$. Assume $c_a=c_{a+1}=i$. Then $a\notin \Des(\ts)\subseteq \Des(S)$, so by Lemma $\ref{thelemma}$ $C_a$ comes after $C_{a+1}$ in the reading word. Similarly, if $c_a=c_{a+1}=i'$ then $C_a$ comes before $C_{a+1}$ in the reading word- so $S$ is indeed the labeling given by the map $St'$.

Lastly we will show that if $T_{S,\beta}$ is a marked peak composition tableau that satisfies $St'(T_{S,\beta})=S$, then $\Des(\beta)\supseteq \Des(S)$. Assume there exists an element $a$ in $\Des(S)\backslash \Des(\beta)$. As $a\notin \Des(\beta)$, we have $|c_a|=|c_{a+1}|=i$ for some $i$. Assume $C_a$ comes before $C_{a+1}$ in the reading word. Then, as $a$ is a descent of $S$, $a+1$ is marked in $S$, and therefore $c_{a+1}=i'$. As we count all the $i'$ before counting any $i$ in the standardization map, $c_a=i'$. As we count primed numbers in the reverse reading order in standardization, we need $a > a+1$ which is a contradiction. The case of $C_a$ coming after $C_{a+1}$ is symmetrical.
\end{proof}

In general, for any set $D \subseteq [n-1]$, the peak set of $D$ is given by: 
\begin{eqnarray*}
\Peak(D)&=&\{i\mid i\in \text{D}\text{ and }i-1 \notin \text{D}\}.
\end{eqnarray*}

Note that the peak compositions of $n$ are exactly those compositions that can be realized as the peak set of a subset $D\subset [n-1]$. In this note we are  mainly interested in the case when $D$ is the descent set of the reading word for a tableau, so we will use the shorthand $\Peak(T)$ to denote $\Peak(\Des(T))$.

\begin{lem}Let $S^0$ denote the standard marked peak composition tableau of shape $\alpha$ that has labels $1,2,\ldots, \alpha_1$ in the first row; labels $\alpha_1+1,\alpha_1+2,\ldots,\alpha_1+\alpha_2$ on the second row and so on. Then  $\comp(\Peak(S_0))=\alpha$. Furthermore, if $S$ is any other standard marked peak composition tableau of shape $\alpha$ , then $\comp(\Peak(S)) < \alpha$ with respect to the lexicographic order. \label{lem:order}
\end{lem}
\begin{proof} First note that by its definition, the peaks of the reading word of $S_0$ are given by the rightmost number in each row except the last, so $\Peak(S_0)=
\{\alpha_1,\alpha_1+\alpha_2,\ldots,\alpha_1+\alpha_2+\cdots+\alpha_{n-1}\}$ and $\comp(\Peak(S_0))=\alpha$. Any change in this ordering results in a tableau $S$ with decreased peak values, so \\$\comp(\Peak(S))<\comp(\Peak(S_0)$) for any $S \neq S_0$.
\end{proof}

Now we are ready to state and prove our main theorem.

\begin{thm} \label{thm:main}The quasisymmetric Schur $P$-functions expand into peak functions as:
\begin{eqnarray*}
\widehat{P}_{\alpha}(X)&=&2^{-\ell(\alpha)}\sum_{S\in\spct(\alpha)}2^{|\Peak(S)|+1}G_{\Peak(S)}(X).
\end{eqnarray*}
In particular, they are positive, integral and unitriangular.
\end{thm}
\begin{proof} This theorem follows from the following observation:
Let $w$ be an unmarked word of length $k$ with $\Peak(\Des(w))=P$. Denote by $M(w)$ the $2^k$ element set of possible marked versions of $w$. Then we have:
\begin{eqnarray}
G_{P}(X)=2^{-|P|-1}\sum_{w' \in M(w)}F_{\Des(w')}(X)
\end{eqnarray} A detailed proof of this can be found in \cite{ezgi}. Plugging in Equation \ref{eqn:defP} we get:
\begin{eqnarray*}
\sum_{S\in\spct(\alpha)}2^{|\Peak(S)|+1}G_{\Peak(S)}(X)=\mathop{\sum\sum}_{\substack{S\in\spct(\alpha)\\w' \in M(\rw(S))}}F_{\Des(w')}(X)=\\\sum_{S'\in\smpc(\alpha)}F_{\Des(S')}(X)=2^{\ell(\alpha)}\widehat{P}_{\alpha}(X).
\end{eqnarray*}
The fact that the expansion is unitriangular follows from Lemma \ref{lem:order}. For integrality, we need to show that for any standard peak composition tableaux $S$ of shape $\alpha$, the number of peaks of $S$ is at least $\ell(\alpha)-1$. Let $S$ be such a tableaux of length $k$. It follows from Definition \ref{def:mpc} Condition (3) that as peak compositions can not have a part of size one except at the last row, for any box in the first column, the box right adjacent has a larger label than the box directly below, so they are ordered as follows: 
\begin{center}
\begin{tikzpicture}[scale=.5]
\draw (0,0) grid (2,-1) (1,-1)--(1,-2)--(0,-2);
\draw[thick] (0,0)--(0,-2);
\draw (0,0)--(-.5,-.5) (0,-1)--(-.5,-1.5) (0,-1.5)--(-.5,-2) (0,-.5)--(-.5,-1);
\draw[dotted] (2,0)--(3,0) (2,-1)--(3,-1) (1,-2)--(2,-2);
\node at (.5,-.5) {$a$};
\node at (.5,-1.5) {$c$} ;
\node at (1.5,-.5) {$b$};
\end{tikzpicture}
\raisebox{.5cm}{\qquad \qquad $\Rightarrow  a<b<c$}
\end{center}
Note that $b$ comes after both $a$ and $c$ in the reading word. If $b$ is not a peak, either $b+1$ or $b-1$ comes after $b$ in the reading word, and is also between $a$ and $c$. Repeating this process, we can find a peak of $S$ between $a$ and $c$. As we can do this for any pair of consecutive entries in the first column, we have at least $\ell(\alpha)-1$ peaks. 
\end{proof}

\section{Further properties of quasisymmetric Schur $Q$-functions}

Equation \eqref{eq:alternating} gives an alternating expansion of Schur's $P$-functions into $\widehat{P}$ functions. A consequence of the negative signs appearing in this expansion is that the quasisymmetric Schur $P$-functions lack some of the positivity properties of their symmetric analogues.

First we will look at the expansions into quasisymmetric Schur functions $\{\RQS_{\alpha}\}_{\alpha}$  \cite{MR3097867}, Young quasisymmetric functions $\{\YQS_{\alpha}\}_{\alpha}$ \cite{MR3097867} and dual immaculate quasisymmetric functions$\{\mathfrak{S}_{\alpha}\}_{\alpha}$ \cite{MR3194160}.

\begin{thm} Quasisymmetric Schur $Q$-functions do not expand positively into quasisymmetric Schur functions or Young quasisymmetric Schur functions.
\end{thm}
\begin{proof} Two examples of minimum size where negative coefficients come up are listed below. 
\begin{multline*}
\widehat{P}_{(2,3)}=\RQS_{(1,4)}-\RQS_{(3,2)}+\RQS_{(1,2,2)}+\RQS_{(2,1,2)}+2\RQS_{(1,3,1)}-\RQS_{(1,1,2,1)}\\+\RQS_{(1,2,1,1)}+\RQS_{(2,1,1,1)},
\end{multline*}

%\begin{multline*}
%\widehat{P}_{(3,2,1)}=\YQS_{(3,2,1)}+\YQS_{(3,1,2)}+\YQS_{(2,3,1)}+\YQS_{(2,2,2)}+\YQS_{(1,4,1)}+\YQS_{(1,2,3)}+\YQS_{(2,2,1,1)}\\+2\YQS_{(1,3,2)}+2\YQS_{(2,1,3)}+\YQS_{(2,1,2,1)}+\YQS_{(2,1,1,2)}-\YQS_{(1,2,2,1)}-\YQS_{(1,2,2,1)}-\YQS_{(1,2,1,1,1)}
%\end{multline*}
\begin{multline*}
\widehat{P}_{(2,3,1)}=\YQS_{(2,3,1)}+\YQS_{(2,2,2)}+\YQS_{(2,2,1,1)}+\YQS_{(2,1,3)}+\YQS_{(2,1,2,1)}+\YQS_{(2,1,1,2)}\\+\YQS_{(1,4,1)}+\YQS_{(1,3,2)}-\YQS_{(1,2,2,1)}-\YQS_{(1,2,1,2)}-\YQS_{(1,2,1,1,1)}.
\end{multline*}
\end{proof}

\begin{cor} Quasisymmetric Schur functions do not expand positively into dual immaculate quasisymmetric functions.
\end{cor}
\begin{proof} This follows as dual immaculate quasisymmetric functions expand positively into young quasisymmetric
schur functions, as shown by Allen, Hallam and Mason in \cite{mason}. A minimal example with negative coefficients is given below:
\begin{eqnarray*}
\widehat{P}_{(3,1)}=\mathfrak{S}_{(3,1)}+\mathfrak{S}_{(2,2)}+\mathfrak{S}_{(2,1,1)}-\mathfrak{S}_{(1,3)}.
\end{eqnarray*}

\end{proof}
%Note that as we have $\widehat{P}_{(3,2,1)}-\widehat{P}_{(2,3,1)}=P_{(3,2,1)}$, their difference has a positive expansion:

%\begin{equation*}
%P_{(3,2,1)}=\YQS_{(3,2,1)}+\YQS_{(3,1,2)}+\YQS_{(2,3,1)}+\YQS_{(2,1,3)}+\YQS_{(1,3,2)}+\YQS_{(1,2,3)}
%\end{equation*}

We will now consider the multiplication of quasisymmetric Schur $P$-functions.

\begin{thm} The multiplication of a Schur's $P$-function with a a quasisymmetric Schur ${P}$-function does not necessarily have a positive expansion into quasisymmetric Schur $P$-functions.
\end{thm}
\begin{proof}Positivity fails even in the simplest case of $\widehat{P}_{(1)}$:
\begin{equation*}
\widehat{P}_{(1)}(X)P_{(3,1)}(X)=\widehat{P}_{(4,1)}(X)+\widehat{P}_{(3,2)}(X)-\widehat{P}_{(2,3)}(X).
\end{equation*}\end{proof}
Note that we do not even have a set inclusion between the diagrams of $(2,3)$ and $(3,1)$ so we do not have a natural rule of calculating the expansion of $\widehat{P}_{(1)}P_{\lambda}$ by adding boxes to the diagram of $\lambda$ and assigning signs. Same statement applies to the multiplication of $\widehat{P}$ functions, as $\widehat{P}_{(3,1)}=P_{(3,1)}$.

\begin{cor} The multiplication of quasisymmetric Schur $P$-functions is not $\widehat{P}$-positive.
\end{cor}

\bibliographystyle{plain}
\bibliography{bib}
\end{document}